\documentclass[12pt]{amsart}
\usepackage{amssymb,amsmath,amsthm,mathrsfs,multirow,xcolor,framed,url}
\usepackage[pdftex,
         pdfauthor={Dandan Chen},
         pdftitle={Congruence modulo 4 for Andrews' integer partition},
         pdfsubject={MATHEMATICS},
         pdfkeywords={Ternary quadratic form, Ramanujan-type congruences, Class number},
         pdfproducer={Latex with hyperref},
         pdfcreator={pdflatex}]{hyperref}
\newtheorem{theorem}{Theorem}[section]
\newtheorem{lemma}[theorem]{Lemma}

\theoremstyle{definition}

\newtheorem{example}[theorem]{Example}

\theoremstyle{remark}

\numberwithin{equation}{section}

\newcommand\nutwid{\overset {\text{\lower 3pt\hbox{$\sim$}}}\nu}

\newcommand\leg[2]{\genfrac{(}{)}{}{}{#1}{#2}} %Legendre symbol
\allowdisplaybreaks
\newcommand{\beqs}{\begin{equation*}}
\newcommand{\eeqs}{\end{equation*}}
\newcommand{\beq}{\begin{equation}}
\newcommand{\eeq}{\end{equation}}
\begin{document}

\title[Congruence modulo 4 for Andrews' integer partition]{Congruence modulo 4 for Andrews' integer partition with even parts below odd parts}

\author{Dandan Chen}
\address{Department of Mathematics, Shanghai University, Shanghai, People’s Republic of China}
\email{mathcdd@shu.edu.cn}
\author{Rong Chen}
\address{School of Mathematical Sciences, Tongji University, Shanghai, People’s Republic of China}
\email{rongchen20@tongji.edu.cn}



\begin{abstract}
We find and prove a class of congruences modulo 4 for Andrews' partition with certain ternary quadratic form. We also discuss distribution of $\overline{\mathcal{EO}}(n)$ and  further prove that $\overline{\mathcal{EO}}(n)\equiv0\pmod4$ for almost all $n$.  This study was inspired by similar congruences modulo 4 in the work by the second author and Garvan.
\end{abstract}

\maketitle
\section{Introduction}
Let $p(n)$ be the number of unrestricted partitions of $n$. Ramanujan discovered and later proved that
\begin{align*}
p(5n+4)\equiv0\pmod5, ~p(7n+5)\equiv0\pmod7 ~~~\text{and}~~ p(11n+6)\equiv0\pmod{11}.
\end{align*}
Congruences like this are called Ramanujan-type congruences. In a recent paper, Andrews \cite{Andrews-Ann-18} studied the partition function $\mathcal{EO}(n)$ which counts the number of partitions of $n$ where every even part is less than each odd part. He denoted by $\overline{\mathcal{EO}}(n)$, the number of
partitions counted by $\mathcal{EO}(n)$ in which \textit{only} the largest even part appears an odd number of times. For example, $\mathcal{EO}(8)=12$ with the relevant partitions being
$8, 6+2, 7+1, 4+4, 4+2+2, 5+3, 5+1+1+1, 2+2+2+2, 3+3+2, 3+3+1+1, 3+1+1+1+1+1, 1+1+1+1+1+1+1+1+1$; and $\overline{\mathcal{EO}}(8)=5$, with the relevant partitions being $8, 4+2+2, 3+3+2, 3+3+1+1, 1+1+1+1+1+1+1+1$.

Andrews proved that the partition function $\overline{\mathcal{EO}}(n)$ has the following generating function \cite[Eq. (3.2)]{Andrews-Ann-18}:
\begin{align}\label{defn-EO-bar}
\sum_{n=0}^{\infty}\overline{\mathcal{EO}}(n)q^n
=\frac{J_4^3}{J_2^2},
\end{align}
where $J_k=\prod_{n=1}^{\infty}(1-q^{kn})$ and $q=\exp{(2\pi i\tau)}$.
In the same paper, he proposed to undertake a more extensive investigation of the properties of $\overline{\mathcal{EO}}(n)$. In \cite{Ray+Barman-NT-20}, Ray and Barman  used the theory of Hecke eigenforms to establish  two infinite families of congruences for $\overline{\mathcal{EO}}(n)$ modulo $2$ and $8$, respectively. 	

In this paper we consider congruences for $\overline{\mathcal{EO}}(n)$ modulo $4$. We use the same method in the second author and Garvan \cite{Rong+Garvan-BAMS} to establish the following stronger theorem, which can immediately deduce the  result in \cite[Theorem 1.1]{Ray+Barman-NT-20}.

\begin{theorem}\label{thm-EO-bar-mod4-0}
Let $k$, $n$ be nonnegative integers. For each $i$ with $1\leq i\leq k+1$, if $p_i\geq 5$ is prime, then for any $j\not\equiv0 \pmod{p_{k+1}}$
\begin{align}
\label{main}
\overline{\mathcal{EO}}\left(p_1^2\cdots p_{k+1}^2n+
\frac{p_1^2\cdots p_{k}^2p_{k+1}(3j+p_{k+1})-1}{3}\right)\equiv0\pmod4,
\end{align}
where $p_{k+1}$ and $j$ satisfy one of the following addtions\\
$\mathrm{(i)}$ $p_{k+1}\not\equiv 7,13\pmod{24}$,\\
$\mathrm{(ii)}$ $\left(\frac{3j}{p_{k+1}}\right)=-1$.
\end{theorem}

\begin{example} Let $k=0$ and $p=5,7$. We have
\begin{align*}
&\overline{\mathcal{EO}}(25n+3)\equiv0\pmod4,~~
\overline{\mathcal{EO}}(25n+13)\equiv0\pmod4,\\
&\overline{\mathcal{EO}}(25n+18)\equiv0\pmod4,~~
\overline{\mathcal{EO}}(25n+23)\equiv0\pmod4,
\end{align*}
and
\begin{align*}
\overline{\mathcal{EO}}(49n+23)\equiv0\pmod4,~~
\overline{\mathcal{EO}}(49n+30)\equiv0\pmod4,~~
\overline{\mathcal{EO}}(49n+44)\equiv0\pmod4.
\end{align*}
\end{example}

To prove this, we will relate $\overline{\mathcal{EO}}(n)$ to certain ternary quadratic form $r_{1,1,3}(n)$, the numbers of representation of $n=x^2+y^2+3z^2$ with $x,y,z\in \mathbb{Z}$. Further we have

\begin{theorem}\label{thm-lim-mod4-0}
$$
\lim_{N\rightarrow \infty}\frac{\#\{n\leq N :\overline{\mathcal{EO}}(n)\equiv 0\pmod{4}\}}{N}=1.
$$
\end{theorem}

The paper is organized as follows. In Section \ref{sec-pre} we introduce some notations and collect some necessary results. In Section \ref{sec-TQF} we discuss related ternary quadratic form.  Then  we give the proofs of Theorem \ref{thm-EO-bar-mod4-0} and \ref{thm-lim-mod4-0} in Section \ref{sec-proof}.
\section{Preliminaries}\label{sec-pre}
In this section, we first collect some useful identities on basic hypergeometric series and class number.

We have the well-known Jacobi's triple product identity, see for example \cite[P. 35]{Berndt-Ramabook-91},
\begin{align*}
\sum_{n=-\infty}^{\infty}z^nq^{n(n-1)/2}=\prod_{n=1}^{\infty}(1-zq^{n-1})(1-z^{-1}q^n)(1-q^n).
\end{align*}
Then we get the two identities
\begin{align}\label{trip-prod}
\sum_{n=-\infty}^{\infty}(-1)^nq^{n^2}=\frac{J_1^2}{J_2} ~~~\text{and}~~~
\sum_{n=-\infty}^{\infty}(-1)^nq^{(3n^2-n)/2}=J_1.
\end{align}

Denote by $r_{1,3,3}(n)$ the number of integral solutions to $x^2+y^2+3z^2=n$. We find the number of integral solutions to $x^2+y^2+3z^2=n$ is equal to the number of integral solutions to $x^2+3y^2+3z^2=3n$,
\begin{align*}
\sum_{n=0}^{\infty}r_{1,3,3}(3n)q^{3n}
=\sum_{x,y,z\in\mathbb{Z} \atop 3|x^2}q^{x^2+3y^2+3z^2}
=\sum_{x,y,z\in\mathbb{Z}}q^{3(3x^2+y^2+z^2)}
=\sum_{n=0}^{\infty}r_{1,1,3}(n)q^{3n}.
\end{align*}

Then according to \cite[Theorem 4.1]{Shemanske},
\begin{lemma}
If $n$ is square-free and $n\equiv 2\pmod {12}$, then
\begin{align}\label{r133-h3n-relation}
r_{1,1,3}(n)=r_{1,3,3}(3n)=2h(-3n).
\end{align}
\end{lemma}

By \cite[P. 191, Corollary]{Shemanske}, for prime $p$ satisfying $(6,p)=1$, we have
 \begin{align}\label{h6p}
 h(-6p)\equiv \begin{cases}
 4\pmod8, &\text{if $p\equiv5,7\pmod8$ },\\
 0\pmod8,&\text{else}.
 \end{cases}
 \end{align}

\begin{theorem}\label{theh} \cite[P.187]{Shemanske} If $n$ is square-free,  then
\begin{align*}
h(-n)=2^{t-1}k,
\end{align*}
where $t$ is the number of distint prime factors of $n$ and $k$ is the number of classes in each genus of $\mathbb{Q}(\sqrt{-n})$.
\end{theorem}

\begin{lemma}\cite[Lemma 2.11]{wang-NT-21}\label{lem-eq-gamma}
Let $A, B$ be positive integers with $A>B$ and $(A,B)=1$. Let
$$
\gamma(N)=\sum_{\begin{smallmatrix}
0\leq n \leq N \\ An+B=m^2p^{4a+1}\\ p \text{ is prime and $p\nmid m$}
\end{smallmatrix}} 1.
$$
Then we have
\begin{align*}
\gamma(N)=\frac{\pi^2}{6}\prod\limits_{p|A}(1+p^{-1})\frac{N}{\log N}+O\left(\frac{N}{\log^2N}  \right).
\end{align*}
\end{lemma}

\section{Related ternary quadratic form}\label{sec-TQF}

Let
$$f(q):=\sum_{n=0}^{\infty}a(n)q^n
:=\left(\sum_{n=-\infty}^{\infty}q^{n^2}\right)\left(\sum_{n=-\infty}^{\infty}q^{3n^2-n}\right)^2.
$$
In this section, we will discuss the properties of $a(n)$ which is related to $\overline{\mathcal{EO}}(n)$ modulo 4 by the class number. The method is similar to the work by the second author and Garvan \cite{Rong+Garvan-BAMS}. Let
$$
F(q):=\sum_{n=0}^{\infty}A(n)q^n:=q^2f(q^{12})=\sum_{x,y,z\in\mathrm{Z}}q^{(6x+1)^2+(6y+1)^2+12z^2}.
$$
So that
\begin{align}\label{a-A-relation}
a(n)=A(12n+2).
\end{align}
Clearly $A(n)=0$ if $n\not\equiv 2\pmod{12}$. We observe that
\begin{align*}
x^2+y^2+3z^2\equiv2\pmod{12}
\end{align*}
if and only if $x^2$ and $y^2$ are congruent to 1 and $3z^2$ congruent to 0 modulo 12. Since $x^2\equiv1\pmod{12}$ if and only if $x\equiv\pm1\pmod6$ and $3z^2\equiv0\pmod{12}$ if and only if $z\equiv0\pmod2$.
\begin{align*}
\sum_{n=0}^{\infty}r_{1,1,3}(12n+2)q^{12n+2}=&\sum_{u,v=\pm1}\sum_{x,y,z\in\mathrm{Z}}q^{(6x+u)^2+(6y+v)^2+3(2z)^2}\\
=&4\sum_{x,y,z\in\mathrm{Z}}q^{(6x+1)^2+(6y+1)^2+12z^2}.
\end{align*}
Therefore
\begin{align}\label{A-r113-relation}
4A(n)=\begin{cases}
r_{1,1,3}(n),&\text{if $n\equiv2\pmod{12}$},\\
0,&\text{else}.
\end{cases}
\end{align}

First, we will discuss the case of $n$ square-free.

\begin{theorem}\label{thm-A-squarefree}
For $n\equiv2\pmod{12}$ is square-free, we have\\
$\mathrm{(i)}$ $A(2)=1$,\\
$\mathrm{(ii)}$ If $n=2p$ for such prime $p\equiv 5,7 \pmod8$, then
\begin{align*}
A(n)\equiv 2\pmod4,
\end{align*}
$\mathrm{(iii)}$ If $n=2p$ for such prime $p\equiv 1,3 \pmod8$ or $n=2\prod_{i=1}^{m}p_i$ for $m\geq2$ with distinct primes $p_i$, then
\begin{align*}
A(n)\equiv 0\pmod4.
\end{align*}
\end{theorem}

\begin{proof}
$\mathrm{(i)}$ By simply calculate, we get $A(2)=1$.\\
$\mathrm{(ii)}$ By \eqref{r133-h3n-relation} and \eqref{A-r113-relation}, we have
\begin{align}\label{A-h3n-relation}
4A(n)=r_{1,1,3}(n)=2h(-3n)~~~ \text{for $n\equiv2\pmod{12}$}.
\end{align}
For $n=2p$ with $p\equiv5,7 \pmod{8}$, by \eqref{h6p}
\begin{align*}
A(n)=A(2p)=\frac12 h(-6p)\equiv 2\pmod4.
\end{align*}
$\mathrm{(iii)}$ For $n=2p$ with $p\equiv1,3 \pmod{8}$, it is easy to know that $A(n)\equiv0\pmod4$ by \eqref{h6p}. When $m\geq2$ and $n=2\prod_{i=1}^{m}p_i$ with $p_i$  distinct primes, applying \eqref{h6p}, \eqref{A-h3n-relation} and Theorem \ref{theh},
\begin{align*}
A(n)=A\left(2\prod_{i=1}^{m}p_i\right)
=\frac12h\left(-6\prod_{i=1}^{m}p_i\right)\equiv0\pmod4.
\end{align*}
\end{proof}

In Theorem \ref{thm-A-squarefree}, we consider the congruence for $A(n)$ modulo 2 and 4 with square-free $n$. Now we remove this restriction. We use the following identity to complete the congruence of $A(n)$.
By \cite{Cooper+Lam}, combining with $r_{1,1,3}(n)=A(n)$ for $n\equiv2\pmod{12}$, we obtain
\begin{align*}
A(p^2n)+\left(\frac{-3n}{p}\right)A(n)+pA(n/{p^2})=(p+1)A(n).
\end{align*}

Using  the same method in \cite{Rong+Garvan-BAMS} and \cite{Rong+Garvan-arxiv}, we have the following lemmas.
\begin{lemma} \label{lem-A-p2-mod2}
Suppose $n\equiv2\pmod{12}$ and $p$ is any prime satisfying $(p,6n)=1$.\\
$\mathrm{(i)}$ $A(p^2n)\equiv A(n)\pmod2$,\\
$\mathrm{(ii)}$ $A(p^3n)$ is always even.
\end{lemma}

\begin{lemma}\label{lem-A-p4-mod2}
Suppose $n\equiv2\pmod{12}$ and $p$ is any prime satisfying $(p,6)=1$. Then $A(p^4n)\equiv A(n)\pmod2$.
\end{lemma}

\begin{lemma}\label{lem-A-p2-mod4}
Suppose $n\equiv2\pmod{12}$ and $p$ is any prime satisfying $(p,6n)=1$, and $A(n)$ is even. Then,\\
$\mathrm{(i)}$ $A(p^2n)\equiv A(n)\pmod4$,\\
$\mathrm{(ii)}$ $A(p^3n)\equiv 0\pmod4$.
\end{lemma}

\begin{lemma}\label{lem-A-p4-mod4}
Suppose $n\equiv2\pmod{12}$ and $p$ is any prime satisfying $(p,6)=1$, and $A(n)$ is even. Then $A(p^4n)\equiv A(n)\pmod4$.
\end{lemma}

\begin{theorem}\label{thm-A-parity}
Suppose $n\equiv2\pmod{12}$.\\
$\mathrm{(i)}$ $A(n)$ is odd if and only if $n$ has the form $n=2m^2$ with $(m,6)=1$,\\
$\mathrm{(ii)}$ $A(n)\equiv 2\pmod 4$ if and only if $n$ has the form
\begin{align*}
n=2p^{4\alpha+1}m^2,
\end{align*}
where $p\equiv 5,7\pmod{8}$ is prime, $m$ and $\alpha\geq0$ are integers with $(m,6p)=1$.
\end{theorem}

\begin{proof}
$\mathrm{(i)}$ ($\Leftarrow$) If $n$ has the form $n=2m^2$ with $(m,6)=1$,  we deduce that
\begin{align*}
A(n)=A(2m^2)\equiv A(2)\equiv1\pmod2;
\end{align*}
($\Rightarrow$) Assume that $A(n)$ is odd. Then $n\equiv2\pmod{12}$ has the prime factorisation
\begin{align*}
n=\prod_{i=1}^{s}p_i^{\alpha_i}.
\end{align*}
By Lemma \ref{lem-A-p4-mod2},
\begin{align*}
A(n)=A(n_1)\pmod2,
\end{align*}
where $n=2n_1t^4$ for some integer $t$ such that
\begin{align*}
n_1=2\prod_{i=1}^{s}p_i^{\beta_i},
\end{align*}
with $1\leq\beta_i\leq3$ for each $i$. By Theorem \ref{thm-A-squarefree} and Lemma \ref{lem-A-p2-mod2}, when there exists an $i$ such that $\beta_i=3$ or $\beta_i=1$, we find that $A(n)$ is even. Hence  $n$ has the form $n=2m^2$ with $(m,6)=1$.

$\mathrm{(ii)}$($\Leftarrow$) If $n=2p^{4\alpha+1}m^2$ and $p\equiv 5,7\pmod{8}$, we deduce that
\begin{align*}
A(n)=A(2p^{4\alpha+1}m^2)\equiv A(2pm^2)\equiv A(2p)\equiv2\pmod4;
\end{align*}

($\Rightarrow$)
 Assume that $A(n)\equiv2\pmod4$.  Then $n\equiv2\pmod{12}$ has the prime factorisation
\begin{align*}
n=\prod_{i=1}^{s}p_i^{\alpha_i}.
\end{align*}
By Lemma \ref{lem-A-p4-mod4},
\begin{align*}
A(n)=A(n_1)\pmod4,
\end{align*}
where $n=n_1t^4$ for some integer $t$ such that
\begin{align*}
n_1=2\prod_{i=1}^{s}p_i^{\beta_i},
\end{align*}
with $1\leq\beta_i\leq3$ for each $i$. By Lemma \ref{lem-A-p2-mod4}, when there exists an $i$ such that $\beta_i=3$ we find that $A(n)\equiv0\pmod4$. So that $1\leq\beta_i\leq2$ and
$$
n_1=2\prod_{i=1}^{k}p_i\prod_{j=k+1}^{s}p_j^2.
$$
Noting that $A(n_1)\equiv 2\pmod 4$ is even, by Lemma \ref{lem-A-p2-mod4}
$$
A(n_1)\equiv A(2\prod_{i=1}^{k}p_i).
$$
Then by Theorem \ref{thm-A-squarefree}, $k=1$ and $p_1\equiv 5,7\pmod8$. This implies $n$ has the form
$$
n=2p^{4\alpha+1}m^2,
$$
with $p\equiv 5,7\pmod8$.
\end{proof}

\section{Proof of Theorem \ref{thm-EO-bar-mod4-0} and \ref{thm-lim-mod4-0}}\label{sec-proof}
From \cite[p.333]{Ray+Barman-NT-20}, we know
\begin{align*}
\sum_{n=0}^{\infty}\overline{\mathcal{EO}}(n)q^n\equiv J_2^2J_4\pmod 4.
\end{align*}

By \eqref{trip-prod}, we have
\begin{align}\label{defn-bn}
\sum_{n=0}^{\infty}b(n)q^n:=J_1^2J_2=\sum_{n=-\infty}^{\infty}(-1)^nq^{n^2}\left(\sum_{n=-\infty}^{\infty}(-1)^nq^{3n^2-n}\right)^2.
\end{align}

\begin{lemma}\label{lem-an-bn-mod4-0}
For any $n\in \mathbb{Z}$, we have
\begin{align*}
a(n)\equiv b(n)\pmod4,
\end{align*}
i.e.
\begin{align*}
\sum_{n=-\infty}^{\infty}q^{n^2}\left(\sum_{n=-\infty}^{\infty}q^{3n^2-n}\right)^2
\equiv \sum_{n=-\infty}^{\infty}(-1)^nq^{n^2}\left(\sum_{n=-\infty}^{\infty}(-1)^nq^{3n^2-n}\right)^2\pmod4.
\end{align*}
\end{lemma}
\begin{proof}
Let $a_1=\sum_{n=-\infty}^{\infty}q^{n^2}$ and $a_2=\sum_{n=-\infty}^{\infty}(-1)^nq^{n^2}$, $b_1=\sum_{n=-\infty}^{\infty}q^{3n^2-n}$ and $b_2=\sum_{n=-\infty}^{\infty}(-1)^nq^{3n^2-n}$.
It is easy to know that $a_1\equiv a_2\pmod4$ and $b_1\equiv b_2\pmod2$.
Hence,
\begin{align*}
a_1b_1^2-a_2b_2^2\equiv a_1b_1^2-a_1b_2^2=4a_1\frac{b_1+b_2}{2}\frac{b_1-b_2}{2}\equiv 0\pmod 4.
\end{align*}
\end{proof}

\begin{theorem}\label{thm-EO-A-mod4}
For any $n\in \mathbb{Z}$, we have
\begin{align*}
\overline{\mathcal{EO}}(n)\equiv A(6n+2)\pmod4.
\end{align*}
\end{theorem}

\begin{proof}
Let $q\rightarrow q^2$ in \eqref{defn-bn}, combining with \eqref{defn-EO-bar}, we have
\begin{align*}
\sum_{n=0}^{\infty}b(n)q^{2n}= J_2^3J_4 \equiv \sum_{n=0}^{\infty}\overline{\mathcal{EO}}(n)q^n \pmod4,
 \end{align*}
 and
 \begin{align*}
 \overline{\mathcal{EO}}(2n)\equiv b(n)\pmod4,~~~~ \overline{\mathcal{EO}}(2n+1)=0.
 \end{align*}
By Lemma \ref{lem-an-bn-mod4-0}, we obtain
\begin{align*}
\overline{\mathcal{EO}}(2n)\equiv b(n)\equiv a(n)=A(12n+2)\pmod4.
\end{align*}
And we know that $\overline{\mathcal{EO}}(2n+1)=0=A(12n+8)$.
\end{proof}

Now we will give the proof of Theorem \ref{thm-EO-bar-mod4-0}.
\begin{proof}
Let $k$, $n$ be nonnegative integers. For each $i$ with $1\leq i\leq k+1$, if $p_i\geq 5$ is prime
\begin{align*}
&\overline{\mathcal{EO}}\left(p_1^2\cdots p_{k+1}^2n+
\frac{p_1^2\cdots p_{k}^2p_{k+1}(3j+p_{k+1})-1}{3}\right)\\
=&A\left(6p_1^2\cdots p_{k+1}^2n+2p_1^2\cdots p_{k}^2p_{k+1}(3j+p_{k+1})\right).
\end{align*}

Let
\begin{align*}
N=&6p_1^2\cdots p_{k+1}^2n+2p_1^2\cdots p_{k}^2p_{k+1}(3j+p_{k+1})\\
=&2p_1^2\cdots p_k^2p_{k+1}(3p_{k+1}n+3j+p_{k+1}).
\end{align*}
Since $(p_{k+1},j)=1$, so that $p_{k+1}\parallel N$. By Theorem \ref{thm-A-parity}, $A(N)$ is even. Assume
\begin{align*}
A(N)=A(2p_1^2\cdots p_k^2p_{k+1}(3p_{k+1}n+3j+p_{k+1}))\equiv2\pmod4
\end{align*}
for such $N$. Using Theorem \ref{thm-A-parity} again, there exists $m$ with $(m,6p_{k+1})=1$ such that $N\equiv 2\pmod{12}$ has the form
\beq
\label{Np}
N=2p_1^2\cdots p_k^2p_{k+1}(3p_{k+1}n+3j+p_{k+1})=p_{k+1}m^2,
\eeq
with $p_{k+1}\equiv 5,7\pmod8$. Further $N\equiv 2\pmod{12}$ implies that $p_{k+1}\equiv 1\pmod6$, so that $p_{k+1}\equiv 7,13\pmod{24}$. Hence \eqref{main} holds under condition $\mathrm{(i)}$. \eqref{Np} implies that $3p_{k+1}n+3j+p_{k+1}$ is a square (not divided by $p_{k+1}$). Therefore $3j\equiv k^2\pmod{p_{k+1}}$ and $\leg{3j}{p_{k+1}}$=1. Hence \eqref{main} holds under condition $\mathrm{(ii)}$. We complete the proof of Theorem \ref{thm-EO-bar-mod4-0}.
\end{proof}

To get the proof of Theorem \ref{thm-lim-mod4-0}, we need the following lemma.

\begin{lemma}
\label{lem-limit}
Let $c>1$ be a constant. If $N$ is sufficiently large, then
\begin{align*}
&\#\{n\leq N :\overline{\mathcal{EO}}(n)\equiv 1\pmod{2}\}\leq \sqrt {6N+1};\\
&\#\{n\leq N :\overline{\mathcal{EO}}(n)\equiv 2\pmod{4}\}
\leq \frac{c\pi^2}{3}\frac{N}{\log {N}}.
\end{align*}
\end{lemma}

\begin{proof}
Using Theorem \ref{thm-EO-A-mod4}, we deduce that
\begin{align*}
\overline{\mathcal{EO}}(2n)=A(12n+2)~~~\quad \text{and}~~~\quad \overline{\mathcal{EO}}(2n+1)=0.
\end{align*}
By Theorem \ref{thm-A-parity}, we have
\begin{align*}
&\#\{n\leq N :A(12n+2)\equiv 1\pmod{2}\}\\
=&\#\{n\leq N :12n+2=2m^2~~ \text{for some integer $m$}\}\\
=&\#\{n\leq N :6n+1=m^2~~\text{for some integer $m$}\}\leq \sqrt {6N+1},
\end{align*}
and
\begin{align*}
&\#\{n\leq N :A(12n+2)\equiv 2\pmod{4}\}\\
=&\#\{n\leq N :12n+2=2p^{4\alpha+1}m^2~\text{for some integers $m$ and $\alpha\geq0$}\}\\
=&\#\{n\leq N :6n+1=p^{4\alpha+1}m^2 ~\text{for some integers $m$ and $\alpha\geq0$}\}.
\end{align*}
Let $(A,B)=(6,1)$ in Lemma \ref{lem-eq-gamma}. We obtain
\begin{align*}
\#\{n\leq N :\overline{\mathcal{EO}}(2n)\equiv 2\pmod{4}\}
=\frac{\pi^2}{3}\frac{N}{\log N}+O\left(\frac{N}{\log^2 N}\right).
\end{align*}
Hence if $N$ is sufficiently large, then
\begin{align*}
\#\{n\leq N :\overline{\mathcal{EO}}(2n)\equiv 2\pmod{4}\}
\leq \frac{c\pi^2}{3}\frac{N}{\log N}.
\end{align*}
Considering the parity of $n$, we find that
\begin{align*}
&\#\{n\leq N :\overline{\mathcal{EO}}(n)\equiv 1\pmod{2}\}\\
=&\#\{2k\leq N :\overline{\mathcal{EO}}(2k)\equiv 1\pmod{2}\}\\
\leq &\#\{k\leq N:\overline{\mathcal{EO}}(2k)\equiv 1\pmod{2}\}\leq \sqrt{6N+1},
\end{align*}
and
\begin{align*}
&\#\{n\leq N :\overline{\mathcal{EO}}(n)\equiv 2\pmod{4}\}\\
=&\#\{2k\leq N :\overline{\mathcal{EO}}(2k)\equiv 2\pmod{4}\}\\
\leq &\#\{k\leq N :\overline{\mathcal{EO}}(2k)\equiv 2\pmod{4}\}\\
\leq &\frac{c\pi^2}{3}\frac{N}{\log {N}}.
\end{align*}
\end{proof}

\begin{proof}[The proof of theorem \ref{thm-lim-mod4-0}]

Note that
\begin{align*}
&\#\{n\leq N :\overline{\mathcal{EO}}(n)\equiv 0\pmod{4}\}\\
=&N-\#\{n\leq N :\overline{\mathcal{EO}}(n)\equiv 1\pmod{2}\}-\#\{n\leq N :\overline{\mathcal{EO}}(n)\equiv 2\pmod{4}\}.
\end{align*}
By Lemma \ref{lem-limit}, we have
\begin{align*}
&\lim_{N\rightarrow \infty}\frac{\#\{n\leq N :\overline{\mathcal{EO}}(n)\equiv 0\pmod{4}\}}{N}\\
=&1-\lim_{N\rightarrow \infty}\frac{\#\{n\leq N :\overline{\mathcal{EO}}(n)\equiv 1\pmod{2}\}}{N}
-\lim_{N\rightarrow \infty}\frac{\#\{n\leq N :\overline{\mathcal{EO}}(n)\equiv 2\pmod{4}\}}{N}\\
=&1.
\end{align*}
We eventually arrive at the desired result.
\end{proof}

\subsection*{Acknowledgements}

The first author was supported by Shanghai Sailing Program (21YF1413600).

\end{document}